\documentclass[a4paper,12pt, twoside, reqno]{amsart}
\usepackage{amssymb}
\usepackage{amsmath}
\newtheorem{theorem}{Theorem}

\newtheorem{defn}{Definition}
\newtheorem{ex}{Example}
\newtheorem{remark}{Remark}[section]

\title{Approximating fixed points of enriched Chatterjea contractions by Krasnoselskij iterative method in Banach spaces}
\author{Vasile BERINDE}
\author{M\u ad\u alina P\u ACURAR}
\begin{document}
\maketitle \pagestyle{myheadings} \markboth{M\u ad\u alina P\u ACURAR} {Approximating fixed points of enriched Chatterjea...}
\begin{abstract}
Using the technique of enrichment of contractive type mappings by Krasnoselskij averaging, introduced in [Berinde, V., {\it Approximating fixed points of enriched nonexpansive mappings by Krasnoselskij iteration in Hilbert spaces}, Carpathian J. Math. {\bf 35} (2019), no. 3, 277-288.], we introduce the class of enriched Chatterjea contractions and prove general fixed point theorems for such contractions in the setting of a Banach space. Examples to illustrate the richness of the new class of contractions and  the relationship between enriched Banach contractions,  enriched Kannan contractions and enriched Kannan contractions are also given. 
\end{abstract}

\section{Introduction}

Let $(X,d)$ be ametric space and  $T:X\rightarrow X$ a mapping. Denote by $Fix\, (T)$ the set of fixed points of $T$, i.e, $Fix\, (T)=\{x\in X: Tx=x\}$. For two distinct points $x,y\in X$, there exist six displacements:
\begin{equation}\label{eq1}
d(Tx, Ty), d(x,y), d(x,Tx), d(y,Ty), d(x,Ty) \textnormal{ and } d(y,Tx).
\end{equation}
In order to obtain a fixed point result for such a mapping $T$, it turned out that two or more displacements in the list \eqref{eq1} are to be used. 

The first metrical fixed point theorem in literature has been established by Banach in \cite{Ban22}, in the setting of (what we call now) a Banach space, and then extended to complete metric spaces by Caccioppoli  \cite{Cac}, and involves the first two of the displacements in \eqref{eq1}:
\begin{equation} \label{eq1a}
d(Tx,Ty)\leq c\cdot d(x,y),\forall x,y \in X,
\end{equation}
where $c$ is a constant, $c\in [0,1)$. 

Picard-Banach contraction mapping principle states that, in a complete metric space $(X,d)$, if $T:X\rightarrow X$ is a Banach contraction, i.e., mapping which satisfies \eqref{eq1a}, then

(i) $Fix\,(T)=\{p\}$  and (ii)  $T^n(x_0) \rightarrow p$ as $n\rightarrow \infty$, for any $x_0$ in $X$.

We remind that, see for example Rus \cite{Rus01}, \cite{Rus03}, a mapping $T$ which satisfies $(i)$ and $(ii)$ above is said to be a {\it Picard operator}.

By \eqref{eq1a} we can see that any Banach contraction is continuous. This fact together with the simplicity and flexibility of the contraction condition \eqref{eq1a} made the contraction mapping principle one the most powerful tools in nonlinear analysis.

In 1968, Kannan \cite{Kan68}, see also \cite{Kan69} established a fixed point theorem which has exactly the same conclusion as Picard-Banach contraction mapping principle  but is based on a contractive condition that involves three displacements from the list \eqref{eq1}:  
\begin{equation} \label{eq2}
d(Tx,Ty)\leq a \left[d(x,Tx)+d(y,Ty)\right],\forall x,y \in X,
\end{equation}
where $a$ is a constant,  $a\in[0,1/2)$. 

It is important to note that a Kannan contraction is in general not continuous, see \cite{Kan68} and \cite{Pac09}, for various examples. It is easy to see that, if $c<\dfrac{1}{3}$, then any Picard-Banach contraction is a Kannan mapping. Indeed, by \eqref{eq1a} and triangle inequality, we have
$$
d(Tx,Ty)\leq c \left[d(x,Tx)+d(Tx,Ty)+d(Ty,y)\right] 
$$
$$
\Leftrightarrow d(Tx,Ty)\leq \frac{c}{1-c} \left[d(x,Tx)+d(y,Ty)\right],
$$
and since $\dfrac{c}{1-c}<\dfrac{1}{2}$ for $c<\dfrac{1}{3}$,  inequality \eqref{eq2} holds for all $x,y\in X$. 

Although both Banach and Kannan contractions are Picard operators, however, the class of Kannan contractions in independent of that of Picard-Banach contractions,  see \cite{Mes} and \cite{Rho}, for a comparison of the main contraction type conditions related to Banach contraction condition  \eqref{eq1a}. Moreover, Banach and Kannan contractions also exhibit a different behaviour with respect to the completeness of the ambient space in the sense that, while Kannan contraction mapping principle characterises the metric completeness, see \cite{Sub}, Banach contraction mapping principle does not, see  \cite{Con}. 

Other fixed point theorems, related to Banach fixed point theorem and Kannan fixed point theorem, have been subsequently established by various authors, see \cite{Ber07}, \cite{Chat}-\cite{Ciric03}, \cite{Rus01}, \cite{Rus08}.  In 1972 \cite{Chat},  Chatterjea  introduced  the following contraction condition:
\begin{equation} \label{eq3}
d(Tx,Ty)\leq b \left[d(x,Ty)+d(y,Tx)\right],\forall x,y \in X,
\end{equation}
where $b$ is a constant,  $b\in[0,1/2)$. 

In fact, all three contraction conditions presented above are independent, see \cite{Mes} and \cite{Rho}. This fact enabled Zamfirescu, in 1972, to formulate a very interesting fixed point theorem that involves all three conditions \eqref{eq1a}, \eqref{eq2}, \eqref{eq3} in an original way:   
\begin{theorem}\label{thZ}
Let $(X,d)$ be a complete metric space and $T:X\longrightarrow X$ a map for which there exist the real numbers $a,b$ and
$c$ satisfying $0\leq a<1$, $0<b, c<1/2$, such that for each pair $x,y$ in $X$, at least one of the following is true:
\newline
\indent $(z_1)$ $d(Tx, Ty)\leq a\,d(x,y)$;\\[4pt]
\indent $(z_2)$ $d(Tx, Ty)\leq b\big[d(x, Tx)+d(y, Ty)\big]$;\\[4pt]
\indent $(z_3)$ $d(Tx, Ty)\leq c\big[d(x,Ty)+d(y,Tx)\big]$.\newline
Then $T$ is a Picard operator.
\end{theorem}
Other important contractions conditions in this family are due to L. B. \' Ciri\' c, see \cite{Cir}-\cite{Cir98}. We give here one of the most general ones:  for all $x,y\in X$,
\begin{equation}  \label{eq4}
d(Tx, Ty)\leq h\max\big\{d(x,y), d(x, Tx), d(y, Ty), d(x, Ty), d(y,Tx)
\big\}, 
\end{equation}
where $0<h<1$. For the impressive rich literature on this area, we refer to the monographs \cite{Ber07}, \cite{Pac09}, \cite{Rus79}-\cite{Rus01}, \cite{Rus08}, and references therein.

On the other hand, the first author, in a very recent paper \cite{Ber19}  introduced the technique of enrichment of contractive mappings, which was then successfully used for the class of strictly pseudo-contractive mappings   \cite{Ber19a}, to Picard-Banach contractions  \cite{Pac19b} and to Kannan contractions  \cite{BerP19}. 

Starting from this background, the aim of this article is to apply the technique of enrichment of contractive type mappings to the class of Chatterjea mappings. For this new class of mappings we prove a fixed point theorem and show that their fixed points can be approximated by means of suitable Krasnoselskij iteration rather than by Picard iteration.  This is the reason why we are working in a Banach space, while most of the fixed point results existing in literature for Chatterjea mappings are stated in the setting of a metric space or of a generalized metric space. 

Examples to illustrate the relationship between enriched Banach contractions and enriched Kannan contractions, by one side, and the class of enriched Chatterjea contractions  are also given. Our results are very general and include as particular cases most of the fixed point results established so far for Chatterjea mappings.

\section{Approximating fixed points of enriched Chatterjea mappings} 

\begin{defn}
Let $(X,\|\cdot\|)$ be a linear normed space. A mapping $T:X\rightarrow X$ is said to be an {\it enriched Chatterjea mapping} if there exist $b\in[0,1/2)$ and $k\in[0,ü\infty)$  such that
$$
\|k(x-y)+Tx-Ty\|\leq  b \left[\|(k+1)(x-y)+y-Ty\|+\right.
$$
\begin{equation} \label{eq3a}
\left.+\|(k+1(y-x)+x-Tx\|\right]\|,\forall x,y \in X.
\end{equation}
To indicate the constants involved in \eqref{eq3a} we shall call  $T$ a $(k,b$)-{\it enriched Chatterjea mapping}. 
\end{defn}

\begin{ex}[ ] \label{ex1}
\indent

All Banach contractions with constant $c<\dfrac{1}{3}$, all Kannan mappings with contraction constant $a<\dfrac{1}{4}$ and all Chatterjea mappings are enriched Chatterjea mapping, i.e.,  they satisfy \eqref{eq3a} with $k=0$. 

Indeed, if $T$ satisfies  \eqref{eq1a} with $c<\dfrac{1}{3}$, then we have
$$
d(Tx,Ty)\leq c\left[d(x,Ty)+d(Ty, Tx)+d(Tx,y)\right]
$$
which yields the inequality
$$
d(Tx,Ty)\leq \frac{c}{1-c} \left[d(x,Ty)+d(Tx,y)\right].
$$
As $c<\dfrac{1}{3}$ implies $\dfrac{c}{1-c}<\dfrac{1}{2}$, this proves the first assertion. 

Similarly, if $T$ is a Kannan mapping satisfying \eqref{eq2} with constant $a<\dfrac{1}{4}$, then we get
$$
d(Tx,Ty)\leq \frac{a}{1-2a} \left[d(x,Ty)+d(Tx,y)\right],
$$ 
with  $\dfrac{a}{1-2a}<\dfrac{1}{2}$. 
\end{ex}

\begin{ex}[ ] \label{ex2}

Let $X=[0,1]$ be endowed with the usual norm and $T:X\rightarrow X$ be defined by $Tx=1-x$, for all $x\in [0,1]$. $T$ is nonexpansive (it is an isometry), $T$ is  not a Banach contraction (see \cite{Pac19b}) or a Kannan mapping (see \cite{BerP19}). Moreover, $T$ is not a  Chatterjea mapping but is an enriched Chatterjea mapping. 

Indeed, if $T$ would be a Chatterjea mapping, then there would exist $b\in[0,1/2)$ such that $\forall x,y \in [0,1]$,
$$
|x-y|\leq 2b \cdot |x+y-1|,
$$
which, for $x=0$ and $y=1$ yields the contradiction $1\leq 0$. 

The enriched Chatterjea  condition \eqref{eq3a} is in this case equivalent to
\begin{equation} \label{eq3b}
|(k-1)(x-y)|\leq b \left[|(k+1)x-(k-1)y-1|+|(k+1)y-(k-1)x-1|\right].
\end{equation} 
Having in view that
$$
2k |x-y|=|[(k+1)x-(k-1)y-1]-[(k+1)y-(k-1)x-1]|
$$
$$
\leq |(k+1)x-(k-1)y-1|+|(k+1)y-(k-1)x-1|,
$$
in order to have \eqref{eq3b} satisfied for all $x,y\in [0,1]$, it is necessary to have $\dfrac{|k-1|}{2k}\leq b $, for a certain $b \in[0,1/2)$.

The only possibility is to have $k<1$ when, by taking $\dfrac{1-k}{2k}= b$ say, one obtains $k=\dfrac{1}{b+2}$. Therefore, for any $b\in[0,1/2)$, $T$ is a $\left(\dfrac{1}{b+2},b\right)$-enriched Chatterjea mapping and $Fix\,(T)=\left\{\dfrac{1}{2}\right\}$.
\end{ex}

\begin{ex}[ ] \label{ex3}
Any of the mappings $T$ in Examples 1.3.4, 1.3.5 and 1.3.7 in P\u acurar \cite{Pac09} are {\it discontinuous enriched Chatterjea mappings}, being simple Chatterjea mappings.
\end{ex}

\begin{remark}
We note that for $T$ in Example \ref{ex2}, Picard iteration $\{x_n\}$ associated to $T$, that is, $x_{n+1}=1-x_n$, $n\geq 0$, does not converge for any $x_0$ different of $\dfrac{1}{2}$, the unique fixed point of $T$. 

This situation is common for nonexpansive mappings and suggests  us the need to consider more elaborate fixed point iterative schemes in order to approximate fixed points of enriched Chatterjea mappings.

\end{remark}

\begin{theorem}  \label{th1}
Let $(X,\|\cdot\|)$ be a Banach space and $T:X\rightarrow X$ a $(k,b$)-{\it enriched Chatterjea mapping}. Then

$(i)$ $Fix\,(T)=\{p\}$;

$(ii)$ There exists $\lambda\in (0,1]$ such that the iterative method
$\{x_n\}^\infty_{n=0}$, given by
\begin{equation} \label{3aa}
x_{n+1}=(1-\lambda)x_n+\lambda T x_n,\,n\geq 0,
\end{equation}
converges to p, for any $x_0\in X$;

$(iii)$ The following estimate holds
\begin{equation}  \label{3.2-1}
\|x_{n+i-1}-p\| \leq\frac{\delta^i}{1-\delta}\cdot \|x_n-
x_{n-1}\|\,,\quad n=0,1,2,\dots;\,i=1,2,\dots
\end{equation}
where $\delta=\dfrac{b}{1-b}$.
\end{theorem}

\begin{proof}
For any $\lambda\in (0,1)$ consider the averaged mapping $T_\lambda$, given by
\begin{equation} \label{eq4}
T_\lambda (x)=(1-\lambda)x+\lambda T(x), \forall  x \in X.
\end{equation}
It easy to prove that $T_\lambda$ possesses the following important property:
$$
Fix(\,T_\lambda)=Fix\,(T).
$$
If $k> 0$ in \eqref{eq3b}, then let us take $\lambda=\dfrac{1}{k+1}$. Obviously, we have $0<\lambda<1$ and thus the contractive condition \eqref{eq3} becomes
$$
\left \|\left(\frac{1}{\lambda}-1\right)(x-y)+Tx-Ty\right\|\leq b \left[\left|\frac{1}{\lambda}(x-y)+y-Ty\right|\right.
$$
$$
\left.+\left |\frac{1}{\lambda}(y-x)+x-Tx\right |\right],\forall x,y \in X,
$$
which is equivalent to
$$
\left \|(1-\lambda)(x-y)+\lambda(Tx-Ty)\right\|\leq b \left[\|x-y+\lambda(y-Ty)\|+\right.
$$
$$
\left.+ \|y-x+\lambda(x-Tx)\|\right],\forall x,y \in X,
$$
and the last one inequality  can be written in a simpler form as
\begin{equation} \label{eq5}
\|T_\lambda (x)-T_\lambda y\|\leq b  \left[\|x-T_\lambda y\|+\|y-T_\lambda x\|\right],\forall x,y \in X,
\end{equation}
with $b\in[0,1/2)$.

The above inequality shows that $T_\lambda$ is a Chatterjea contraction in the sense of \eqref{eq3}. 

According to \eqref{eq4}, the iterative process $\{x_n\}^\infty_{n=0}$ defined by  \eqref{3aa} is the Picard iteration associated to $T_\lambda$, that is,
$$
x_{n+1}=T_\lambda x_n,\,n\geq 0.
$$ 
Take $x=x_n$ and $y=x_{n-1}$ in  \eqref{eq5} to get
$$
\|x_{n+1}-x_{n}\|\leq b \left(\|x_n-x_{n}\|+\|x_{n-1}-x_{n+1}\|\right)
$$
$$
\leq b\left(\|x_{n-1}-x_{n}\|+\|x_{n}-x_{n+1}\|\right),
$$
which yields
$$
\|x_{n+1}-x_{n}\|\leq \frac{b}{1-b}\|x_{n}-x_{n-1}\|,\,n\geq 1.
$$
Since $0<b<\dfrac{1}{2}$, by denoting $\delta=\dfrac{b}{1-b}$, we have $0<\delta<1$ and therefore the sequence $\{x_n\}^\infty_{n=0}$ satisfies
\begin{equation} \label{eq6}
\|x_{n+1}-x_{n}\|\leq \delta \|x_{n}-x_{n-1}\|,\,n\geq 1.
\end{equation}
By \eqref{eq6} one obtains routinely the following two estimates
\begin{equation} \label{eq7a}
\|x_{n+m}-x_{n}\|\leq \delta^n \cdot \frac{1-\delta^m}{1-\delta}\cdot \|x_{1}-x_{0}\|,\,n\geq 0, m\geq 1.
\end{equation}
and
\begin{equation} \label{eq8}
\|x_{n+m}-x_{n}\|\leq \delta \cdot \frac{1-\delta^m}{1-\delta}\cdot \|x_{n}-x_{n-1}\|,\,n\geq 1, \,m\geq 1.
\end{equation}
Now, by \eqref{eq7a} it follows that $\{x_n\}^\infty_{n=0}$ is a Cauchy sequence and hence it is convergent in the Banach space $(X,\|\cdot\|)$. Let us denote
\begin{equation} \label{eq10a}
p=\lim_{n\rightarrow \infty} x_n.
\end{equation}
We first prove that $p$ is a fixed point of $T_\lambda$. We have
\begin{equation} \label{eq9}
\|p-T_\lambda p\|\leq \|p-x_{n+1}\|+\|x_{n+1}-T_\lambda p\|=\|x_{n+1}-p\|+\|T_\lambda x_{n}-T_\lambda p\|.
\end{equation}
By \eqref{eq5} it results that 
$$
\|T_\lambda x_{n}-T_\lambda p\|\leq b \left[\|x_{n}-T_\lambda p\|+\|p-T_\lambda x_n\|\right],
$$
and therefore, by \eqref{eq9} one obtains
$$
\|p-T_\lambda p\|\leq (b+1)\|x_{n+1}-p\|+b\left[\|x_{n}-p\|+\|p-T_\lambda p\|\right],
$$ 
which finally yields
\begin{equation} \label{eq10}
\|p-T_\lambda p\|\leq \frac{b+1}{1-b} \cdot \|x_{n+1}-p\|+\delta \|x_{n+1}-p\|,\,n\geq 0.
\end{equation}
Now, by letting  $n\rightarrow \infty$ in \eqref{eq10} we get $\|p-T_\lambda p\|=0$, that is, $p=T_\lambda p$. So, $p\in Fix\,(T_\lambda)$. 

We now prove that $p$ is the unique fixed point of $T_\lambda$. Assume that $q\neq p$ is another fixed point of $T_\lambda$. Then, by \eqref{eq5} 
$$
\|p-q\|\leq 2b \cdot \|p-q\|,
$$
which yields the contradiction $1\leq 2b<1$. Hence $Fix\,(T_\lambda)=\{p\}$ and since $Fix\,(T)=Fix(\,T_\lambda)$, claim $(i)$ is proven. 

Conclusion $(ii)$ now follows by \eqref{eq10a}.

To prove $(iii)$, we let $m\rightarrow \infty$ in \eqref{eq7a}  and \eqref{eq8} to get
\begin{equation} \label{eq11}
\|x_n-p\|\leq  \frac{\delta^n}{1-\delta}\cdot \|x_{1}-x_{0}\|,\,n\geq 1
\end{equation}
and
\begin{equation} \label{eq12}
\|x_n-p\|\leq \frac{\delta}{1-\delta}\cdot \|x_{n}-x_{n-1}\|,\,n\geq 1,
\end{equation}
respectively. Now we can merge \eqref{eq11}  and \eqref{eq12} to get the unifying error estimate \eqref{3.2-1}.

The remaining case $k=0$ is similar to $k\neq 0$ with the only difference that in this case $\lambda=1$ and hence we work with $T=T_1$, when Kasnoselskij iteration \eqref{3aa} reduces to the simple Picard iteration
$$
x_{n+1}=T x_n, n\geq 0.
$$
\end{proof}

\begin{remark}
1) It is well known, see for example Berinde \cite{Ber07}, that any Chatterjea mappings is a strictly quasi contractive mapping, that is
$$
\|Tx-p\|\leq \theta \cdot \|x-p\|,\forall x \in X, p\in Fix\,(T),\,(0<\theta <1).
$$
On the other hand, $T$ in Example \ref{ex2} is nonexpansive, which is not strictly quasi contractive. So, enriched Chatterjea mappings forms a larger class of mappings than the class of strictly quasi contractive mappings.

2) In the particular case $k=0$, by Theorem \ref{th1} we get the classical Chatterjea fixed point theorem in the setting of a Banach space.

\end{remark}

\section{Approximating fixed points of enriched Chatterjea type mappings} 

In the renown Rhoades' classification of contractive conditions \cite{Rho}, Banach contraction condition \eqref{eq1a} is numbered $(1)$,   Kannan contraction condition \eqref{eq2} is numbered $(4)$, while Chatterjea contraction condition \eqref{eq3} is numbered $(11)$. The next  contraction condition, numbered $(12)$, is the following one: 
\begin{equation}  \label{eq4a}
d(Tx, Ty)\leq h\max\big\{d(x, Ty), d(y,Tx)\big\}, \forall x,y \in X.
\end{equation}
where $0<h<1$. 

It is clear that Chatterjea contraction condition \eqref{eq3} implies \eqref{eq4a} but the reverse is not true as shown by the next example. 

\begin{ex}[ ] \label{ex2b}
\indent
Let $X=[0,1]$ be endowed with the usual norm and $T:X\rightarrow X$ be defined by $Tx=0$, if $x\in [0,1)$ and $T(1)=\dfrac{1}{2}$. Then $T$ satisfies \eqref{eq4a} but does not verify \eqref{eq3}. Indeed, for $x\in [0,1)$ and $y=1$ condition \eqref{eq4a} reduces to
$$
\left|0-\dfrac{1}{2}\right|\leq h \max\big\{|x-\dfrac{1}{2}|,|1-0|\big\}=h, 
$$ 
which is true for any $h$ satisfying  $\dfrac{1}{2}\leq h<1$. On the other hand, suppose $T$ satisfies Chatterjea contraction condition \eqref{eq3}. Then, by taking $x=\dfrac{1}{2}$ and $y=1$ in \eqref{eq3} we get
$$
\left|0-\frac{1}{2}\right|\leq b\left[\left|\frac{1}{2}-\frac{1}{2}\right|+|1-0|\right] \Leftrightarrow \frac{1}{2}\leq b<\dfrac{1}{2},
$$
a contradiction.
\end{ex}
So, the previous example motivates us to consider the next definition.

\begin{defn} \label{def2}
Let $(X,\|\cdot\|)$ be a linear normed space. A mapping $T:X\rightarrow X$ is said to be an {\it enriched Chatterjea type mapping} if there exist $k\in[0,ü\infty)$ and $h\in[0,1)$ such that
$$
\|k(x-y)+Tx-Ty\|\leq h \max\left\{\|(k+1)(x-y)+y-Ty\|,\right.
$$
\begin{equation} \label{eq13}
\left. \|(k+1)(y-x)+x-Tx\|\right\},\forall x,y \in X.
\end{equation}
To indicate the constants involved in \eqref{eq13}, we shall also call  $T$ as a $(k,h$)-{\it enriched Chatterjea type  mapping}. 
\end{defn}

\begin{ex}[ ] \label{ex2a}
\indent

(1) If $k=0$ then by \eqref{eq13}, we obtain the original Chatterjea type contractions. Hence, any Chatterjea type  contraction is a $(0,h$)-enriched Chatterjea type  contraction.

(2) Any $(k,a$)-enriched Chatterjea contraction is a $(k,h$)-enriched Chatterjea type  contraction, with $h=2 a$, in view of the inequality $u+v\leq \max\{u,v\}$. This implies that the nonexpansive map given in Example \ref{ex2} is a $(2(1-a),2a)$-enriched Chatterjea type  contraction, for any $a\in (0,1/2)$. For this mapping, as shown in the previous section, Picard iteration does not converge, in general. 

\end{ex}
Fixed points  of strictly {\it enriched} Chatterjea type  mappings can be approximated by  Krasnoselskij iterative method, as shown by the next theorem.

\begin{theorem}  \label{th2}
Let $(X,\|\cdot\|)$ be a Banach space and $T:X\rightarrow X$ a $(k,h$)-{\it enriched Chatterjea type mapping}. Then

$(i)$ $Fix\,(T)=\{p\}$;

$(ii)$ There exists $\lambda\in (0,1]$ such that the iterative method
$\{x_n\}^\infty_{n=0}$, given by
\begin{equation} \label{eq1.3a}
x_{n+1}=(1-\lambda)x_n+\lambda T x_n,\,n\geq 0,
\end{equation}
converges to p, for any $x_0\in X$;
\end{theorem}

\begin{proof}
Like in the proof of Theoem \ref{th1}, for any $\lambda\in (0,1)$, we consider the averaged mapping $T_\lambda$, given by
\begin{equation} \label{eq1.4}
T_\lambda (x)=(1-\lambda)x+\lambda T(x), \forall  x \in X,
\end{equation}
which is known to have the  property $
Fix(\,T_\lambda)=Fix\,(T).
$

If $k> 0$ in \eqref{eq13}, then let us denote  $\lambda=\dfrac{1}{k+1}\in (0,1)$. Thus the contractive condition \eqref{eq13} becomes
$$
\left \|\left(\frac{1}{\lambda}-1\right)(x-y)+Tx-Ty\right\|\leq h \max\left\{\left \|\frac{1}{\lambda}(x-y)+y-Ty\right \|,\right.
$$
$$
\left. \left \|\frac{1}{\lambda}(y-x)+x-Tx\right \|\right\},\forall x,y \in X.
$$
which can be written in an equivalent form as
$$
\left \|(1-\lambda)(x-y)+\lambda(Tx-Ty)\right\|\leq h \max\left\{ \|x-y+\lambda(y-Ty)\|,\right.
$$
$$
\left. \|y-x+\lambda(x-Tx)\|\right\},\forall x,y \in X.
$$
The last inequality expresses the fact that $T_\lambda$ is a Chatterjea type contraction, i.e.,
\begin{equation} \label{eq1.5}
\|T_\lambda x-T_\lambda y\|\leq h  \max\left\{\|x-T_\lambda x\|,\|y-T_\lambda y\|\right\},\forall x,y \in X.
\end{equation}
Now, apply \' Cir\' c fixed point theorem for quasi-contractions (see \cite{Cir74}) to get the conclusion.

The remaining case $k=0$ is also similar to $k\neq 0$ with the only difference that now $\lambda=1$ and hence we shall work with $T=T_1$, when Kasnoselskij iteration \eqref{eq3a} reduces to the simple Picard iteration, We apply directly to $T$ \' Cir\' c fixed point theorem for quasi-contractions (see \cite{Cir74}).
\end{proof}

\section{Conclusions}


\begin{thebibliography}{99}

\bibitem{Ariz} Ariza-Ruiz, D., Jim\' enez-Melado, A. and L\' opez-Acedo, G., {\it A fixed point theorem for weakly Zamfirescu mappings}, Nonlinear Anal. {\bf 74} (2011), no. 5, 1628--1640.

\bibitem{Ban22} Banach, S., {\it Sur les op\' erations dans les ensembles abstraits et leurs applications aux \' equations int\' egrales}, Fund Math. {\bf 3}  (1922), 133--181.

\bibitem{Ban32} Banach, S., \textit{Theorie des Operations Lineaires}. Monografie Matematyczne, Warszawa-Lwow, 1932.

\bibitem{BB07} Berinde, M. and Berinde, V., \textit{On a general class of multi-valued weakly Picard mappings}. J. Math. Anal. Appl. 326 (2007), no. 2, 772--782.

\bibitem {Berinde-CGen} Berinde, V., \emph{Contrac\c{t}ii generalizate \c{s}i aplica\c{t}ii}, Editura Cub Press 22, Baia Mare, 1997.	

\bibitem{Ber03} Berinde, V., \textit{On the approximation of fixed points of weak contractive mappings}. Carpathian J. Math. {\bf 19} (2003), no. 1, 7--22.

\bibitem{Ber03a} Berinde, V., \textit{Approximating fixed points of weak $\phi$-contractions using the Picard iteration}. Fixed Point Theory {\bf 4} (2003), no. 2, 131--142.

\bibitem{Ber04} Berinde, V., \textit{Approximating fixed points of weak contractions using the Picard iteration}, Nonlinear Anal. Forum {\bf 9} (2004), no. 1, 43--53.

\bibitem {Ber04M} Berinde, V., \textit{A common fixed point theorem for nonself mappings}, Miskolc Math. Notes \textbf{5} (2004), No. 2, 137--144.

\bibitem {MB} Berinde, V.,  \textit{Approximation of fixed points of some nonself generalized $\phi$-contractions}. Math. Balkanica (N.S.) 18 (2004), no. 1-2, 85--93

\bibitem{Ber05a} Berinde, V., \textit{A convergence theorem for some mean value fixed point iteration procedures}. Demonstratio Math. 38 (2005), no. 1, 177--184.

\bibitem{Ber07} Berinde, V., \textit{Iterative Approximation of Fixed Points}, Springer, 2007.

\bibitem{Ber07a} Berinde, V., \textit{A convergence theorem for Mann iteration in the class of Zamfirescu operators}. An. Univ. Vest Timi\c s. Ser. Mat.-Inform. 45 (2007), no. 1, 33--41.

\bibitem{Ber08} Berinde, V., \textit{General constructive fixed point theorems for \' Ciri\' c-type almost contractions in metric spaces}. Carpathian J. Math. {\bf 24} (2008), no. 2, 10--19.

\bibitem{Ber09a} Berinde, V., \textit{Approximating common fixed points of noncommuting discontinuous weakly contractive mappings in metric spaces}. Carpathian J. Math. 25 (2009), no. 1, 13--22.

\bibitem{Ber09} Berinde, V., \textit{Some remarks on a fixed point theorem for \' Ciri\' c-type almost contractions}. Carpathian J. Math. 25 (2009), no. 2, 157--162.

\bibitem{Ber10} Berinde, V., \textit{Approximating common fixed points of noncommuting almost contractions in metric spaces}. Fixed Point Theory 11 (2010), no. 2, 179--188.

\bibitem{Ber10a} Berinde, V., \textit{Common fixed points of noncommuting discontinuous weakly contractive mappings in cone metric spaces}. Taiwanese J. Math. 14 (2010), no. 5, 1763--1776.

\bibitem{Ber10b} Berinde, V., \textit{Common fixed points of noncommuting almost contractions in cone metric spaces}. Math. Commun. 15 (2010), no. 1, 229--241.


\bibitem{Ber12} Berinde, V., \textit{Approximating fixed points of implicit almost contractions}. Hacet. J. Math. Stat. 41 (2012), no. 1, 93--102.

\bibitem{Ber19}  Berinde, V., {\it Approximating fixed points of enriched nonexpansive mappings by Krasnoselskij iteration in Hilbert spaces}, Carpathian J. Math. {\bf 35} (2019), no. 3, 277-288. 

\bibitem{Ber19a}  Berinde, V., {\it Approximating fixed points of enriched strictly pseudocontractive operators in Hilbert spaces}  (submitted)


\bibitem{Ber05} Berinde, V., Berinde, M., \textit{On Zamfirescu's fixed point theorem}. Rev. Roumaine Math. Pures Appl. 50 (2005), no. 5-6, 443--453.


\bibitem{BerMR} Berinde, V., M\u aru\c ster, \c St. and Rus, I. A., \textit{An abstract point of view on iterative approximation of fixed points of nonself operators}. J. Nonlinear Convex Anal. {\bf 15} (2014), no. 5, 851--865.


\bibitem{BerP08} Berinde, V. and P\u acurar, M., \textit{Fixed points and continuity of almost contractions}. Fixed Point Theory 9 (2008), no. 1, 23--34.

\bibitem{BerP09} Berinde, V. and P\u acurar, M., \textit{A note on the paper "Remarks on fixed point theorems of Berinde'' [MR2404199]}. Nonlinear Anal. Forum {\bf 14} (2009), 119--124.

\bibitem{BerP13} Berinde, V. and P\u acurar, M., \textit{Fixed point theorems for nonself single-valued almost contractions}. Fixed Point Theory 14 (2013), no. 2, 301--311.

\bibitem{Pac19b} Berinde, V., P\u acurar, M., {\it Approximating fixed points of enriched contractions in Banach spaces} (submitted)

\bibitem{BerP19}  Berinde, V. and P\u acurar, M., {\it Fixed point theorems of Kannan type mappings with applications to split feasibility  and variational inequality problems}  (submitted)

\bibitem{BPet} Berinde, V., Petric, M., {\it Fixed point theorems for cyclic non-self single-valued almost contractions}, Carpathian J. Math. {\bf 31} (2015), no. 3, 289--296.

\bibitem{Bia} Bianchini, R. M. T., {\it Su un problema di S. Reich riguardante la teoria dei punti fissi}, Boll. Un. Mat. Ital. {\bf 5} (1972), 103--108.

\bibitem{Bor} Borcut, M., P\u acurar, M, and Berinde, V., {\it Tripled fixed point theorems for mixed monotone Chatterjea type contractive operators}. J. Comput. Anal. Appl. {\bf 18} (2015), no. 5, 793--802.

\bibitem{Cac} Caccioppoli, R., {\it Un teorema generale sull'esistenza di elementi uniti in una transformazione funzionale}, Rend Accad dei Lincei. {\bf 11}  (1930), 794--799.

\bibitem{Chat} Chatterjea, S. K., {\it Fixed-point theorems}, C. R. Acad. Bulgare Sci. {\bf 25} (1972), 15--18.

\bibitem{Chat76} Chatterjea, S. K. Sur un th\' eor\` eme du point fixe.  Math. Balkanica {\bf 6} (1976), 30--32 (1978).

\bibitem{Cir} \' Ciri\' c, L. B., {\it On a family of contractive maps and fixed points}, Publ. Inst. Math. (Beograd) (N.S.) {\bf 17(31)} (1974), 45--51.

\bibitem {Cir71} \' Ciri\' c, Lj. B., \textit{Generalized contractions and fixed-point theorems}, Publ. l'Inst. Math. (Beograd) 12 (1971), 19--26.

\bibitem {Cir74} \' Ciri\' c, Lj. B., \textit{A generalization of Banach's contraction principle}, Proc. Am. Math. Soc. 45 (1974), 267--273.

\bibitem {Cir99} \' Ciri\' c, Lj. B., \textit{Convergence theorems for a sequence of Ishikawa iteration for nonlinear quasi-contractive mappings}, Indian J. Pure Appl. Appl. Math. 30 (4) (1999), 425--433.

\bibitem {Ciric71} \' Ciri\' c, Lj. B., \textit{Generalized contractions and fixed-point theorems}, Publ. l'Inst. Math. (Beograd) 12 (1971) 19--26.

\bibitem {Cir71a} \' Ciri\' c, Lj. B., \emph{On contraction type mappings}, Math. Balkanica 1 (1971), 52--57.

\bibitem {Ciric93} \' Ciri\' c, Lj. B., \textit{A remark on Rhoades' fixed point theorem for non-self mappings},  Internat. J. Math. Math. Sci. 16 (1993), no. 2, 397--400.

\bibitem {Cir98} \' Ciri\' c, Lj. B., \textit{Quasi contraction non-self mappings on Banach spaces}, Bull. Cl. Sci. Math. Nat. Sci. Math. No. 23 (1998), 25--31.

\bibitem {Ciric03} \' Ciri\' c, Lj. B., Ume, J. S., Khan, M. S. and Pathak, H. K., \textit{On some nonself mappings}, Math. Nachr. 251 (2003), 28--33.

\bibitem {Ciric-Presic}  Ciri\' c, L.B.,  Presi\' c, S.B., \emph{On Presi\' c type generalization of the Banach contraction mapping principle}, Acta Math. Univ. Comenianae, \textbf{76} (2007), No. 2, 143--147.

\bibitem{Ciric11}  \' Ciri\' c, L., Abbas, M., Saadati, R. and Hussain, N., \textit{Common fixed points of almost generalized contractive mappings in ordered metric spaces}. Appl. Math. Comput. 217 (2011), no. 12, 5784--5789.

\bibitem{Con} Connell, E. H., {\it Properties of fixed point spaces}, Proc. Amer. Math. Soc. {\bf 10} (1959) 974--979.

\bibitem{Enj} Enjouji, Y., Nakanishi, M., Suzuki, T., {\it A generalization of Kannan's fixed point theorem}, Fixed Point Theory Appl. {\bf 2009}, Art. ID 192872, 10 pp.

\bibitem{Fish} Fisher, B., {\it A fixed point theorem}, Math. Mag. {\bf 48} (1975), no. 4, 223--225.

\bibitem{Gab} Gabeleh, M. and Shahzad, N., {\it Approximate optimal solutions and generalized contractions in the sense of Chatterjea}. Politehn. Univ. Bucharest Sci. Bull. Ser. A Appl. Math. Phys. {\bf 77} (2015), no. 4, 45--56.

\bibitem{Kan68} Kannan, R, {\it Some results on fixed points}, Bull. Calcutta Math. Soc. {\bf 60}  (1968), 71--76.

\bibitem{Kan69} Kannan, R., {\it Some results on fixed points. II}, Amer. Math. Monthly {\bf 76} (1969), 405--408.

\bibitem{Kik08} Kikkawa, M., Suzuki, T., {\it Some similarity between contractions and Kannan mappings}, Fixed Point Theory Appl. {\bf 2008}, Art. ID 649749, 8 pp.

\bibitem{Kik08a} Kikkawa, M., Suzuki, T., {\it Some similarity between contractions and Kannan mappings. II}, Bull. Kyushu Inst. Technol. Pure Appl. Math. No. 55 (2008), 1--13. 

\bibitem{Kik17} Kikkawa, M., Suzuki, T., {\it Fixed point theorems for \' Ciri\' c type contractions and others in complete metric spaces}. Linear Nonlinear Anal. {\bf 3} (2017), no. 1, 111--120.

\bibitem{Koh} Kohsaka, F. and Suzuki, T., {\it Existence and approximation of fixed points of Chatterjea mappings with Bregman distances}, Linear Nonlinear Anal. {\bf 3} (2017), no. 1, 73--86.

\bibitem{Mal} Malceski, A., Ibrahimi, A. and Malceski, R., {\it Extending Kannan and Chatterjea theorems in $2$-Banach spaces by using sequentialy convergent mappings}.  Mat. Bilten {\bf 40} (2016), no. 1, 29--36.

\bibitem {Mes}  J. Meszaros, {\it A comparison of various definitions of contractive type mappings}, Bull. Calcutta Math. Soc. {\bf 84} (1992), no.2, 167--194.

\bibitem{Mor} Morales, J. R. and Rojas, E. M., {\it Geraghty's approach for Kannan, Chatterjea and Branciari mappings in $b$-metric spaces}. Indian J. Math. {\bf 59} (2017), no. 1, 73--106.

\bibitem{Nak} Nakanishi, M., Suzuki, T., {\it An observation on Kannan mappings}, Cent. Eur. J. Math. {\bf 8} (2010), no. 1, 170--178.

\bibitem{Pac09} P\u acurar, M., {\it Iterative methods for fixed point approximation}, Editura Risoprint, Cluj-Napoca, 2009.

\bibitem{Pac14} P\u acurar, M.,  Berinde, V.,  Borcut, M. and Petric, M., {\it Triple fixed point theorems for mixed monotone Pre\v si\' c-Kannan and Pre\v si\' c-Chatterjea mappings in partially ordered metric spaces}. Creat. Math. Inform. {\bf 23} (2014), no. 2, 223--234.

\bibitem {Pi890} Picard, E., \textit{Memoire sur la th\' eorie des equations aux deriv\' ees partielles et la methode des approximations successives}. J. Math. Pures et Appl., \textbf{6} (1890), 145--210.

\bibitem{Raz} Razani, A. and Parvaneh, V., {\it Some fixed point theorems for weakly $T$-Chatterjea and weakly $T$-Kannan-contractive mappings in complete metric spaces}, Russian Math. (Iz. VUZ) {\bf 57} (2013), no. 3, 38--45.

\bibitem{Rho} Rhoades, B. E., {\it A comparison of various definitions of contractive mappings}, Trans. Amer. Math. Soc. {\bf 226} (1977), 257--290.

\bibitem {Rhoa83} B.E. Rhoades, {\it Contractive definitions revisited}, Contemporary Mathematics {\bf 21} (1983) 189--205.

\bibitem {Rhoa88} B.E. Rhoades, {\it Contractive definitions and continuity}, Contemporary Mathematics {\bf 72} (1988) 233--245.

\bibitem {Rus79} Rus, I.A., {\it Principles and Applications of the Fixed Point Theory} (in Romanian), Editura Dacia, Cluj-Napoca, 1979.

 \bibitem {Rus79a} Rus, I.A., \emph{Metrical Fixed Point Theorems}, Univ. of Cluj-Napoca, 1979.

\bibitem {Rus83} Rus, I.A., \emph{Generalized contractions}, Seminar on Fixed Point Theory {\bf 3} (1983) 1--130.

 \bibitem{Rus01} Rus, I. A., \textit{Generalized Contractions and Applications}, Cluj Univ. Press, Cluj-Napoca, 2001.

\bibitem{Rus03} Rus, I. A., \textit{Picard operators and applications}, Sci. Math. Jpn., \textbf{58} (2003), No. 1, 191--219.

\bibitem{Rus14} Rus, I. A., \textit{Heuristic introduction to weakly Picard operator theory}, Creat. Math. Inform., \textbf{23} (2014), No. 2, 243--252.

\bibitem{Rus13}  Rus, I. A., \textit{Five open problems in fixed point theory in terms of fixed point structures (I): Singlevalued operators}, in Fixed Point Theory and Its Applications, Editors R. Espinola, A. Petru\c sel and S. Prus, Casa C\u ar\c tii de \c Stiin\c t\u a, Cluj-Napoca, 2013, pp. 39-60.

\bibitem{Rus08} Rus, I. A., Petru\c sel, A. and Petru\c sel, G., \textit{Fixed Point Theory}, Cluj Univ. Press, Cluj-Napoca, 2008.

\bibitem{Rus13} Rus, I. A. and \c Serban, M.-A., \textit{Basic problems of the metric fixed point theory and the relevance of a metric fixed point theorem}, Carpathian J. Math., \textbf{29} (2013), No. 2, 239--258.

\bibitem{Shi} Shioji, N., Suzuki, T., Takahashi, W., {\it Contractive mappings, Kannan mappings and metric completeness}, Proc. Amer. Math. Soc. {\bf 126} (1998), no. 10, 3117--3124.

\bibitem{Sub} Subrahmanyam, P. V., {\it Remarks on some fixed-point theorems related to Banach's contraction principle}, J. Mathematical and Physical Sci. {\bf 8} (1974), 445--457; errata, ibid. {\bf 9} (1975), 195.

\bibitem{Suz05} Suzuki, T., {\it Contractive mappings are Kannan mappings, and Kannan mappings are contractive mappings in some sense}, Comment. Math. (Prace Mat.) {\bf 45} (2005), no. 1, 45--58.

\bibitem{Wlo} Wlodarczyk, K., Plebaniak, R., {\it Kannan-type contractions and fixed points in uniform spaces}, Fixed Point Theory Appl. {\bf 2011}, 2011:90, 24 pp.

\bibitem {Zam} Zamfirescu, T., \textit{Fix point theorems in metric spaces} Arch. Math. (Basel), {\bf 23} (1972), 292--298.



\end{thebibliography}
\end{document}